\documentclass[preprint]{elsarticle}
\usepackage{amsbsy}
\usepackage{amsfonts}
\usepackage{amsmath}
\usepackage{amssymb}

\newdefinition{definition}{Definition}
\newtheorem{theorem}{Theorem}

\newproof{proof}{Proof}
\newtheorem{corollary}{Corollary}
\newtheorem{lemma}{Lemma}

\newdefinition{remark}{Remark}
\newdefinition{question}{Question}

\newdefinition{example}{Example}

\begin{document}

\begin{frontmatter}
\title{On groups generated by bi-reversible automata: the two-state case over a changing alphabet}
\author[rvt]{Adam Woryna}
\ead{adam.woryna@polsl.pl}
\address[rvt]{Silesian University of Technology, Institute of Mathematics, ul. Kaszubska 23, 44-100 Gliwice, Poland}

\begin{abstract}
The notion of an automaton over a changing alphabet $X=(X_i)_{i\geq 1}$ is used to define and study automorphism groups of the tree $X^*$ of finite words over $X$. The concept of bi-reversibility for Mealy-type automata is extended to automata over a changing alphabet. It is proved that a non-abelian free group can be generated by a two-state bi-reversible automaton over a changing alphabet $X=(X_i)_{i\geq 1}$ if and only if $X$ is  unbounded. The characterization of groups generated by a two-state bi-reversible automaton over the sequence of binary alphabets is established.
\end{abstract}


\end{frontmatter}

\section{Transducers and groups generated by them}

A changing alphabet is simply an infinite sequence $X=(X_i)_{i\geq 1}$ of non-empty finite sets $X_i$ (so-called sets of letters or finite alphabets). If the sequence $(|X_i|)_{i\geq 1}$ of cardinalities of the finite alphabets is bounded, then $X$ is called bounded, otherwise unbounded.

An automaton $A$ (other names: transducer, time-varying automaton) over a changing alphabet $X=(X_i)_{i\geq 1}$ is a finite set $Q$ (set of states) together with an infinite sequence $\varphi=(\varphi_i)_{i\geq 1}$ of transition functions $\varphi_i\colon Q\times X_i\to Q$ and an infinite sequence $\psi=(\psi_i)_{i\geq 1}$ of output functions $\psi_i\colon Q\times X_i\to X_i$. If all these sequences  are constant, then $A$ is called a Mealy-type automaton and it is usually  identified with the quadruple $(X_1, Q, \varphi_1, \psi_1)$. If for each $i\geq 1$ and $q\in Q$ the mapping $\sigma_{i,q}\colon x\mapsto \psi_i(q,x)$ ($x\in X_i$) is a permutation of the alphabet $X_i$, then $A$ is called invertible. The mapping $\sigma_{i,q}$ ($i\geq 1$, $q\in Q$) is called the labeling of a state $q$ in the $i$-th transition of the automaton $A$.

It is convenient to interpret an automaton $A=(X, Q, \varphi, \psi)$ as a machine, which being at any moment $i\geq 1$ in a state $q\in Q$
and reading from the input tape a letter $x\in X_i$, types on the output tape the letter $\psi_i(q, x)$, goes to the state $\varphi_i(q, x)$ and moves both tapes to the next position. This interpretation naturally associates with  each state $q\in Q$ the transformation $A_q\colon X^*\to X^*$ of the tree $X^*$ of finite words over the changing alphabet $X=(X_i)_{i\geq 1}$. The tree $X^*$ consists  of  finite sequences $x_1x_2\ldots x_t$ ($t\geq 1$) such that $x_i\in X_i$ for  every $1\leq i\leq t$ (we also assume the empty word denoted by $\epsilon$); if the sequence $(|X_i|)_{i\geq 1}$ is constant, then we obtain a so-called regular rooted tree. The mapping $A_q\colon X^*\to X^*$ is defined as follows: $A_q(\epsilon):=\epsilon$ and
$$
A_q(x_1x_2\ldots x_t):=\psi_1(q_1,x_1)\psi_2(q_2, x_2)\ldots\psi_t(q_t,x_t),
$$
where the states $q_i\in Q$ are defined recursively $q_1:=q$ and $q_{i+1}:=\varphi_i(q_i,x_i)$ for $1\leq i<t$.

If an automaton $A=(X, Q, \varphi, \psi)$ is invertible, then  the inverse automaton $A^{-1}:=(X, Q, \varphi', \psi')$ is defined as follows:
$$
\varphi'_i(q, x):=\varphi_i(q, \sigma_{i,q}^{-1}(x)),\;\;\;\psi'_i(q,x):=\sigma_{i,q}^{-1}(x)
$$
for all $i\geq 1$, $q\in Q$ and $x\in X_i$. In this case the mappings $A_q$ and $A^{-1}_q$ ($q\in Q$) are mutually inverse automorphisms of the tree $X^*$; that is, they are mutually inverse permutations of the set $X^*$ preserving the empty word and the vertex adjacency (we assume that two words  are adjacent if one of them arises from the second by deleting the last letter). In particular, these permutations preserve the lengths and the common beginnings of words.  We then refer  to the group generated by the permutations $A_q$ for $q\in Q$ (with the composition of mappings as a product) as the group generated by the automaton $A$ and denote it by $G(A)$, i.e., we define
$$
G(A):=\langle A_q\colon q\in Q\rangle.
$$
Each group of the form $G(A)$ is finitely generated; and it is an example of a  residually finite group (as the whole group $Aut(X^*)$ is residually finite -- see~\cite{9}).

The potential for Mealy-type transducers  was discovered in group theory about fifty years ago, when realized that quite simple formulae describing the transition and output functions in an automaton $A$ may result in some  exotic properties of the group $G(A)$. The flagship example is the famous Grigorchuk group generated by a five-state Mealy automaton over the binary alphabet (for more on  interesting properties of groups generated by Mealy automata see the survey paper \cite{9} or the monograph~\cite{10}). The notion of an  automaton over a changing alphabet was introduced in \cite{15}, where we obtained a useful  combinatorial tool to define and study  automorphism groups of a spherically homogeneous rooted tree, which is not necessarily a regular rooted tree.

An important problem in the theory of groups generated by transducers is to verify which finitely generated abstract groups $G$ can be realized as groups of the form $G(A)$, as well as to find  simple and applicable formulae describing the corresponding automaton $A$. In particular, the problem of finding  an explicit realization of a non-abelian free group  turned out to be  far from trivial. It was  solved by Glasner and Mozes (\cite{1}) in 2005, but even in 80s of the last century it was conjectured that the so-called Aleshin-Vorobets automaton generates a non-abelian free group of rank three, which M. Vorobets and Ya. Vorobets finally confirmed in \cite{2}.

The question on the existence of a 2-state automaton over a bounded changing alphabet which generates a non-abelian free group  is still open. In particular, we do not know if there is a 2-state Mealy automaton generating $\mathcal{F}_2$ (non-abelian free group of rank two). Indeed, in all known realizations of  non-abelian free groups by Mealy automata the generating automata have  more than two states (\cite{1,10,13,14}). Apart from Mealy automata, there are two various realizations of $\mathcal{F}_2$ by a 2-state automaton over an unbounded changing alphabet (see Examples~\ref{ex1}--\ref{ex2} in the next section). It turns out that all these transducers are bi-reversible.

\section{Bi-reversible automata over a changing alphabet}\label{sec2}

The concept of reversibility and bi-reversibility for Mealy automata was introduced by Macedo\'{n}ska, Nekrashevych and Sushchanskyy (\cite{3}), who found a connection between the group  of  automorphisms  defined by bi-reversible automata over a finite alphabet $X$ and the commensurator of the non-abelian free group $\mathcal{F}_{|X|}$. Further development was due to Glasner and Mozes (\cite{1}), who associate with each bi-reversible Mealy automaton a square complex and its universal covering. This allowed  to construct the first examples of Mealy automata generating non-abelian free groups.

Let us recall that a Mealy automaton $A=(X, Q, \varphi, \psi)$ is reversible if  every letter $x\in X$ acts (via the transition function $\varphi$)  like a permutation of the set  of states, i.e. for each $x\in X$ the mapping $q\mapsto \varphi(q,x)$ ($q\in Q$)  defines a permutation of the set $Q$. If $A$ is invertible and both the automata $A$ and $A^{-1}$ are reversible, then $A$ is called bi-reversible.

Recently, there has been a lot of interest in this kind of Mealy automata.  For example, Bondarenko, D'Angeli and Rodaro (\cite{4}) constructed a bi-reversible 3-state Mealy automaton over the ternary alphabet $X=\{0,1,2\}$ which  generates the lamplighter group $\mathbb{Z}_3\wr\mathbb{Z}$, providing  the first example of a not finitely presented group generated by a bi-reversible Mealy automaton. In \cite{6} Klimann dealt with the semigroups generated by  2-state reversible Mealy automata and showed that every  such a semigroup is either finite or free. In \cite{5} Godin and Klimann referred to the well-known Burnside problem and proved that connected reversible Mealy automata with a prime number of states can  not generate infinite torsion groups. D'Angeli and Rodaro (\cite{7}) associated with each bi-reversible Mealy automaton $A$ an automaton $(\partial A)^-$,  which they called the enriched dual of $A$, and next, they showed how the boundary dynamics of the semigroup generated by $(\partial A)^-$ characterizes the algebraic properties (in particular, the property of being not free) of the group $G(A)$.

It turns out that the notion of reversibility and bi-reversibility can be naturally extended to  automata over a changing alphabet.

\begin{definition}
We call an automaton $A=(X, Q, \varphi, \psi)$ over a changing alphabet $X=(X_i)_{i\geq 1}$  reversible if for every $i\geq 1$ and every letter $x\in X_i$ the mapping $q\mapsto \varphi_i(q,x)$ ($q\in Q$)  defines a permutation of the set $Q$. If $A$ is invertible and both $A$ and $A^{-1}$ are  reversible automata, then we call the automaton $A$  bi-reversible.
\end{definition}

An obvious example of bi-reversible automata are the invertible automata $A=(X, Q, \varphi, \psi)$ in which  the transition functions are defined in the diagonal way, i.e. the equality  $\varphi_i(q,x)=q$ holds for all $i\geq 1$, $q\in Q$ and $x\in X_i$. We  showed in \cite{8} that the automata of  diagonal type provide a universal construction for all finitely generated residually finite groups.

\begin{theorem}(\cite{8}, Theorem~1)\label{t0}
For any unbounded changing alphabet $X=(X_i)_{i\geq 1}$ and for any finitely generated residually-finite  group $G$ with an arbitrary finite generating set $S$, there is an invertible automaton $A=(X, S, \varphi, \psi)$ of  diagonal type such that the group $G(A)$ generated by  this automaton is isomorphic to $G$.
\end{theorem}

In the case of a bounded changing alphabet $X=(X_i)_{i\geq 1}$ any invertible and diagonally defined automaton $A$ over $X$ generates a finite group. Indeed, if $A=(X, Q, \varphi, \psi)$ is of diagonal type and  $\sigma_{i,q}\in Sym(X_i)$  is the labeling of the state $q\in Q$ in the $i$-th transition ($i\geq 1$) of $A$, then the mapping $A_q\mapsto (\sigma_{i,q})_{i\geq 1}$ ($q\in Q$) induces an    embedding of $G(A)$ into the Cartesian product $\prod_{i\geq 1} Sym(X_i)$. In the case of a bounded changing alphabet this product embeds into an infinite direct power of a finite symmetric group, which is a locally finite group (see, for example, Lemma~1 in~\cite{16}). Since $G(A)$ is finitely generated, it must be a finite group in this case.

The automaton $A$ from Theorem~\ref{t0} can be constructed in  such a way that  the following self-similarity property holds (see also \cite{8}): for each $i\geq 0$ the mapping $A_q\mapsto A^{(i)}_q$ ($q\in Q$) induces an isomorphism  $G(A)\simeq G(A^{(i)})$, where $A^{(i)}$ denotes the $i$-shift of $A$, i.e. the automaton $(X^{(i)}, Q, \varphi^{(i)}, \psi^{(i)})$ such that  $X^{(i)}:=(X_j)_{j>i}$, $\varphi^{(i)}:=(\varphi_j)_{j>i}$ and $\psi^{(i)}:=(\psi_j)_{j>i}$. Nevertheless, the problem with the diagonal realization from  Theorem~\ref{t0}  is that  its existence can be easily deduced from  the assumption that the group  $G$ is residually finite. This implies that the formula for the output functions in this realization is highly nonconstructive (see the proof of Theorem~1 in \cite{8}).

\begin{example}\label{ex1}
In \cite{12}, we discovered a more explicit realization of $\mathcal{F}_2$ by an  automaton $A$ of diagonal type with the 2-element set $Q:=\{q_1, q_2\}$ of states,  which works  over an arbitrary, unbounded changing alphabet $X=(X_i)_{i\geq 1}$. To this end, in the set of freely reduced group words in the symbols $a$ and $b$, we provided a lexicographic order $\prec$ induced by
$$
\mbox{\it the empty word}\prec a\prec a^{-1}\prec b\prec b^{-1}.
$$
Next, we constructed  two permutations of the set $\mathbb{N}:=\{1,2,\ldots\}$, also denoted by  $a$ and $b$, with the following property: if $W:=W(a,b)$ is the group word in the $n$-th position ($n\geq 1$) in the above ordering, then the permutation $W\colon \mathbb{N}\to\mathbb{N}$ maps   $1$ into $n$. Unfortunately, the  formulae for $a$ and $b$ turned out to be quite complicated, as we obtained:
$$
\begin{array}{l}
a(n):=\left\{
\begin{array}{lcl}
  2, & {\rm if} & n=1, \\
  n+4\cdot3^k, & {\rm if}  & 2\cdot3^k\leq n<3^{k+1},\;\;k\geq 0, \\
  n-2\cdot3^k, &{\rm if}  & 3^{k+1}\leq n<4\cdot 3^k,\;\;k\geq 0, \\
  n+3^{k+1}, &{\rm if} & 4\cdot3^k\leq n<2\cdot3^{k+1},\;\;k\geq 0,
\end{array}
\right.\\
b(n):=\left\{
\begin{array}{lcl}
  4, &{\rm if} & n=1, \\
  n+10\cdot3^k, &{\rm if} & 2\cdot3^k\leq n<5\cdot3^k,\;\;k\geq 0, \\
  n-\lfloor13\cdot3^{k-1}\rfloor, &{\rm if} & 5\cdot3^k\leq n<17\cdot 3^{k-1},\;\;k\geq 0, \\
  n-4\cdot 3^k, &{\rm if} & 17\cdot3^{k-1}\leq n<2\cdot3^{k+1},\;\;k\geq 0.
\end{array}
\right.
\end{array}
$$
Now, if we  assume that $X_i:=\{1, \ldots, r_i\}$ for some  $r_i\geq 1$ ($i\geq 1$), then  the output functions in the automaton $A$ can be defined as follows:
$$
\psi_i(q_1, x):=\left\{
\begin{array}{ll}
  a(x),&{\rm if}\; x\in a^{-1}(X_i), \\
  a_i(x),&{\rm if}\;x\notin a^{-1}(X_i),
\end{array}
\right.\;\;
\psi_i(q_2,x):=\left\{
\begin{array}{ll}
b(x),&{\rm if}\; x\in b^{-1}(X_i),\\
b_i(x),&{\rm if}\; x\notin b^{-1}(X_i),
\end{array}
\right.
$$
where $a_i, b_i$ ($i\geq 1$) are any bijections of the form:
$$
a_i\colon X_i\setminus a^{-1}(X_i)\rightarrow X_i\setminus a(X_i),\;\;\;\;\;b_i(x)\colon X_i\setminus b^{-1}(X_i)\rightarrow X_i\setminus b(X_i).
$$
In this way, we obtained a 2-state bi-reversible automaton $A$ of  diagonal type for which $G(A)\simeq \mathcal{F}_2$ (in \cite{12}, we proved this isomorphism in the case $X_i=\{1,\ldots, i\}$ but the proof works in the general case).
\end{example}

It would be interesting to simplify the output functions from Example~\ref{ex1} and obtain a realization of the  group $\mathcal{F}_2$ by a  2-state automaton of  diagonal type.  Such a simplification  is justified by the fact that for any infinite collection $\mathcal{C}$ of non-abelian finite simple groups the free group $\mathcal{F}_2$ is residually $\mathcal{C}$ (\cite{17}). Hence, it follows that  for every infinite sequence $(G_i)_{i\geq 1}$ of non-abelian finite simple groups with the unbounded sequence $(|G_i|)_{i\geq 1}$ of their orders it is possible to choose  the 2-element generating sets $\{\alpha_i, \beta_i\}$ of $G_i$ ($i\geq 1$) such that the sequences $\alpha=(\alpha_i)_{i\geq 1}$ and $\beta=(\beta_i)_{i\geq 1}$ generate in the Cartesian product $\prod_{i\geq 1} G_i$ a group isomorphic to $\mathcal{F}_2$. Nevertheless, there is not known (according  to our knowledge) any explicit construction of such generating   sets.

\begin{example}\label{ex2}
Let $Y=(Y_i)_{i\geq 1}$ be an arbitrary changing alphabet such that  $|Y_i|\geq 2$ for every $i\geq 1$. For each $i\geq 1$ let us choose two letters $x^0_i, x^1_i\in Y_i$ and let $\tau_i,\pi_i\in Sym(Y_i)$ be two permutations of the set $Y_i$ such that $\tau_i$ is an arbitrary transposition and $\pi_i$ is an arbitrary long cycle satisfying    $\tau_i(x^0_i)=\pi_i(x^0_i)=x^1_i$. In particular, the 2-element set $\{\tau_i, \pi_i\}$ constitutes a standard  generating set of the symmetric group $Sym(Y_i)$. Let $A=(Y, Q, \varphi, \psi)$ be an automaton in which $Q:=\{q_1, q_2\}$ and the transition and output functions are defined as follows:
$$
\varphi_i(q, x)=\left\{
\begin{array}{lcl}
q_2, &{\rm if}&x=x^0_i,\;q=q_1,\\
q_1,&{\rm if}& x=x^0_i,\; q=q_2,\\
q, &{\rm if}&  x\neq x^0_i,
\end{array}
\right.\;\;\;
\psi_i(q, x)=\left\{
\begin{array}{lcl}
\pi_i(x),&{\rm if} &q=q_1,\\
\tau_i(x),&{\rm if} &q=q_2.
\end{array}
\right.
$$
Obviously, the automaton $A$ is invertible and in the inverse automaton $A^{-1}=(Y, Q, \varphi', \psi')$, we have:
$$
\varphi'_i(q, x)=\left\{
\begin{array}{lcl}
q_2, &{\rm if}&x=x^1_i,\;q=q_1,\\
q_1,&{\rm if}& x=x^1_i,\; q=q_2,\\
q, &{\rm if}&  x\neq x^1_i,
\end{array}
\right.\;\;\;
\psi'_i(q, x)=\left\{
\begin{array}{lcl}
\pi_i^{-1}(x),&{\rm if} &q=q_1,\\
\tau_i^{-1}(x),&{\rm if} &q=q_2.
\end{array}
\right.
$$
In particular,  the automaton $A$ is bi-reversible. In \cite{11}, we have shown that if the changing alphabet $Y=(Y_i)_{i\geq 1}$ is unbounded and the sequence $(|Y_i|)_{i\geq 1}$ is non-decreasing, then the group $G(A)=\langle A_{q_1}, A_{q_2}\rangle$ is a non-abelian free group of rank 2 generated freely by the automorphisms $A_{q_1}$ and $A_{q_2}$. We believe the requirement that  the sequence $(|Y_i|)_{i\geq 1}$ is non-decreasing is unnecessary, but in \cite{11}, we used it  to prove the main result (see the proof of Proposition~7 and the proof of Theorem~1 on  p. 6431 therein).
\end{example}

Comparing the above two automaton realizations of the group $\mathcal{F}_2$, it seems that the one from Example~\ref{ex2} is more  handy to study the geometric action of $\mathcal{F}_2$ on the corresponding tree. We now use this realization to show that the action may be level-transitive (i.e. any two words of equal length belong to the same orbit).

\begin{theorem}
Let $Y=(Y_i)_{i\geq 1}$ be a changing alphabet such that $|Y_i|\geq 3$ for every $i\geq 1$ and the inequality $i\neq j$ implies ${\rm GCD}(|Y_i|-1, |Y_j|-1)=1$. Then the group $G(A)$ generated by the automaton $A$ from Example~\ref{ex2} acts level-transitively on the tree $Y^*$.
\end{theorem}
\begin{proof}
Let us denote
$$
a_i:=A^{(i-1)}_{q_1},\;b_i:=A^{(i-1)}_{q_2},\;c_i:=a_ib^{-1}_i,\;\;\;i\geq 1,
$$
where $A^{(i)}=(Y^{(i)}, Q, \varphi^{(i)}, \psi^{(i)})$ ($i\geq 0$) is the $i$-shift of $A$. For all $i\geq 1$, $x\in Y_{i}$ and $w\in (Y^{(i)})^*$ we obtain by the definition of $A$ the following recursions:
$$
a_i(xw)=\left\{
\begin{array}{ll}
\pi_i(x)b_{i+1}(w), &{\rm if}\;x=x^0_i,\\
\pi_i(x)a_{i+1}(w), &{\rm if}\;x\neq x^0_i,
\end{array}
\right.b_i(xw)=\left\{
\begin{array}{ll}
\tau_i(x)a_{i+1}(w), &{\rm if}\;x=x^0_i,\\
\tau_i(x)b_{i+1}(w), &{\rm if}\;x\neq x^0_i.
\end{array}
\right.
$$
Consequently, we have
$$
b^{-1}_i(xw)=\left\{
\begin{array}{ll}
\tau_i(x)a^{-1}_{i+1}(w), &{\rm if}\;x=x^1_i,\\
\tau_i(x)b^{-1}_{i+1}(w), &{\rm if}\;x\neq x^1_i,
\end{array}
\right.
$$
and hence
\begin{eqnarray*}
c_i(xw)=a_ib^{-1}_i(xw)&=&\left\{
\begin{array}{ll}
a_i(\tau_i(x)a^{-1}_{i+1}(w)), &{\rm if}\;x=x^1_i\\
a_i(\tau_i(x)b^{-1}_{i+1}(w)), &{\rm if}\;x\neq x^1_i
\end{array}
\right.=\\
&=&\left\{
\begin{array}{ll}
\sigma_i(x)c^{-1}_{i+1}(w), &{\rm if}\;x=x^1_i,\\
\sigma_i(x)c_{i+1}(w), &{\rm if}\;x\neq x^1_i,
\end{array}
\right.
\end{eqnarray*}
where $\sigma_i:=\pi_i\tau_i^{-1}$ for every $i\geq 1$. The permutation $\sigma_i\in Sym(Y_i)$ is a cycle of length $|Y_i|-1$ and  $x^1_i\in Y_i$ is the only letter which is  fixed by this permutation. Thus, we have $\sigma_i(x^1_i)=x^1_i$ and for every $x\in Y_i\setminus\{x^1_i\}$ and every integer $N$ the equality $\sigma^N_i(x)=x$ implies the divisibility $|Y_i|-1\mid N$. In particular, for all $i\geq1$, $x\in Y_i$, $w\in (Y^{(i)})^*$ and every integer $N$, we obtain the following formula:
\begin{equation}\label{jc}
c_i^N(xw)=\left\{
\begin{array}{ll}
xc^{-N}_{i+1}(w), &{\rm if}\;x=x^1_i,\\
\sigma_i^N(x)c^N_{i+1}(w), &{\rm if}\;x\neq x^1_i.
\end{array}
\right.
\end{equation}

Let $v=y_1y_2\ldots y_t\in Y^*$ be an arbitrary word. It is enough to show that there is  $g\in G(A)=\langle a_1, b_1\rangle$ such that $v=g(w)$, where $w:=x^1_1x^1_2\ldots x^1_t$. If we denote $w_1:=b_1(w)$, then by the recursion for $b_i$, we have $w_1=x^0_1x^0_2\ldots x^0_t$ and hence, by the recursion (\ref{jc}), we obtain:
$$
c^N_1(w_1)=\sigma_1^N(x^0_1)\sigma^N_2(x^0_2)\ldots\sigma^N_t(x^0_t),\;\;\;N\geq 1.
$$
Let us denote
$$
I:=\{1\leq i\leq t\colon y_i=x^1_i\},\;\;\;I':=\{1,2,\ldots, t\}\setminus I.
$$
Since for each $1\leq i\leq t$ we have $|Y_i|\geq 3$ and for all $1\leq i,j\leq t$ the inequality $i\neq j$ implies that the numbers $|Y_i|-1$ and $|Y_j|-1$ are coprime, we obtain  by the Chinese Remainder  Theorem that there is $N_0\geq 1$ such that $\sigma_i^{N_0}(x^0_i)=x^0_i$ for every $i\in I$ and $\sigma_i^{N_0}(x^0_i)\notin\{x^0_i, x^1_i\}$ for every $i\in I'$.
Thus, if we denote
$$
w_2:=c_1^{N_0}(w_1)=\sigma_1^{N_0}(x^0_1)\sigma_2^{N_0}(x^0_2)\ldots\sigma_t^{N_0}(x^0_t),
$$
then  by the inequalities $\sigma_i^{N_0}(x^0_i)\neq x^1_i$ ($1\leq i\leq t$) and by the  recursion for $b_i^{-1}$, we obtain:
$b_1^{-1}(w_2)=z_1z_2\ldots z_t$, where  $z_i:=\tau_i\sigma^{N_0}(x^0_i)$ for every $1\leq i\leq t$. Hence, we obtain: $z_i=x^1_i$ for every  $i\in I$ and $z_i\neq x^1_i$ for every $i\in I'$. Next,  for every integer $N$, we obtain by (\ref{jc}):
$$
 c_1^Nb_1^{-1}(w_2)=c_1^N(z_1z_2\ldots z_t)=\sigma_1^{E_{N,1}}(z_1)\sigma_2^{E_{N,2}}(z_2)\ldots \sigma_t^{E_{N,t}}(z_t),
$$
where for every $1\leq i\leq t$ we have: $E_{N,i}=N$ or $E_{N,i}=-N$ depending on the parity of the number of elements of the set $I\cap\{1,\ldots, i\}$. Note that for each  $i\in I$ we have: $\sigma_i^{E_{N,i}}(z_i)=\sigma_i^{E_{N,i}}(x^1_i)=x^1_i=y_i$. Since for every  $i\in I'$ we  have: $z_i\neq x^1_i$ and $y_i\neq x^1_i$,  we can use again our assumption and the Chinese Remainder Theorem to obtain that there is $N_1\geq 1$ such that $\sigma_i^{E_{N_1,i}}(z_i)=y_i$ for every $i\in I'$. Hence, if we denote
$$
w_3:=c_1^{N_1}b_1^{-1}(w_2)=\sigma_1^{E_{N_1,1}}(z_1)\sigma_2^{E_{N_1,2}}(z_2)\ldots \sigma_t^{E_{N_1,t}}(z_t),
$$
then we have $w_3=y_1y_2\ldots y_t=v$. Thus for the  element $g:=c_1^{N_1}b_1^{-1}c_1^{N_0}b_1\in G(A)$, we obtain $v=g(w)$.\qed
\end{proof}

\begin{remark}
Note that if $|Y_1|=|Y_2|=2$, then the action of $G(A)$ is not level-transitively, as the  restriction of this action  to the subtree $Y_1\times Y_2\subseteq Y^*$ coincides with the action of the cyclic group of order two.
\end{remark}

One could naturally ask what happens, if we set the transparent output functions from Example~\ref{ex2} to a 2-state automaton of  diagonal type. If
 $A$ is such an automaton and for every $i\geq 2$ we take in the symmetric group $Sym(i)$ of the set $\{1,2,\ldots, i\}$ the transposition $\tau_i:=(1,2)$ and the long cycle $\pi_i:=(1,2,\ldots, i)$, then there exists a subset $I\subseteq \mathbb{N}$ such that the group $G(A)$ is  isomorphic to the group
$$
G_I:=\langle \tau_I, \pi_I\rangle\leq \prod_{i\in I}Sym(i),
$$
where $\tau_I:=(\tau_i)_{i\in I}$ and $\pi_I:=(\pi_i)_{i\in I}$. The groups of the form $G_I$ ($I\subseteq \mathbb{N}$) were studied in~\cite{18,19,15}. By Example~4.1 in \cite{18}, it follows that  all these groups are amenable, and hence we can not realize $\mathcal{F}_2$ in this way. However, the same  reasoning as in~\cite{15} (see Section~5 therein) shows that if the set $I$ is infinite (which corresponds to the case when the changing alphabet $Y$ is unbounded), then  the semigroup generated by $\tau_I$ and $\pi_I$ is free. In particular, the group $G_I$ is of exponential growth in this case. Further, it follows by Proposition~4.1 in \cite{19} that if  $I, I'\subseteq \mathbb{N}\setminus\{1,2,3,4\}$ and $I\neq I'$, then the groups $G_I$ and $G_{I'}$ are not isomorphic. As a result, we obtain the following corollary.

\begin{corollary}
For an arbitrary unbounded changing alphabet $X=(X_i)_{i\geq 1}$ there are uncountably many pairwise non-isomorphic groups of the form $G(A)$, where $A$ is a 2-state bi-reversible automaton of  diagonal type over $X$.
\end{corollary}

\section{On generation of $\mathcal{F}_2$ by  2-state bi-reversible automata}

In view of the above results, it is interesting to verify for which changing alphabets $X=(X_i)_{i\geq 1}$    there is a 2-state bi-reversible automaton over $X$  which  generates $\mathcal{F}_2$, as well as to provide for every such  alphabet $X$  an explicit and naturally defined  realization of  $\mathcal{F}_2$ by a suitable automaton over $X$. Here we give the solution to both of these problems.

\begin{theorem}\label{t1}
Let $X=(X_i)_{i\geq 1}$ be an arbitrary changing alphabet.  Then the free group $\mathcal{F}_2$ can be generated by a 2-state bi-reversible automaton over $X$ if and only if $X$ is unbounded. On the other hand, if $X$ is bounded and $A=(X, \{q_1, q_2\}, \varphi, \psi)$ is an arbitrary 2-state bi-reversible automaton, then the element $A_{q_1}^{-1}A_{q_2}\in G(A)$ is of finite order, and, additionally, if this element is trivial (i.e. $A_{q_1}=A_{q_2}$), then the group $G(A)$ is finite. In particular, there is no 2-state bi-reversible Mealy automaton generating $\mathcal{F}_2$.
\end{theorem}
\begin{proof}
Let $A=(X, Q, \varphi, \psi)$ be an arbitrary bi-reversible automaton with the 2-element set $Q:=\{q_1, q_2\}$ of states. For each $i\geq 1$ let us define the subsets $Z_i, T_i\subseteq X_i$ as follows:
\begin{equation}\label{e2}
Z_i:=\{x\in X_i\colon \varphi_i(q_1, x)=q_1\},\;\;\;T_i:=\{x\in X_i\colon \varphi_i(q_1, x)=q_2\}.
\end{equation}
Obviously, we have $Z_i\cup T_i=X_i$ and $Z_i\cap T_i=\emptyset$. Since $A$ is reversible, we also have
\begin{equation}\label{eee3}
Z_i=\{x\in X_i\colon \varphi_i(q_2, x)=q_2\},\;\;\;T_i=\{x\in X_i\colon \varphi_i(q_2, x)=q_1\}.
\end{equation}
Let $A^{-1}=(X, Q, \varphi', \psi')$ be the inverse automaton of $A$. For every $i\geq 1$ let us  denote
$$
\alpha_i:=\sigma_{i,q_1}\in Sym(X_i),\;\;\;\beta_i:=\sigma_{i, q_2}\in Sym(X_i).
$$
Then by the definition of $A^{-1}$ and by (\ref{e2}), we obtain for each $i\geq 1$ and each   $x\in X_i$:
\begin{equation}\label{e1}
\varphi'_i(q_1,x)=\varphi_i(q_1, \sigma^{-1}_{i,q_1}(x))=\varphi_i(q_1, \alpha_i^{-1}(x))=
\left\{
\begin{array}{ll}
q_1,&{\rm if}\;x\in \alpha_i(Z_i),\\
q_2,&{\rm if}\;x\in \alpha_i(T_i).
\end{array}
\right.
\end{equation}
Similarly, we obtain by (\ref{eee3})
\begin{equation}\label{e3}
\varphi'_i(q_2,x)=
\left\{
\begin{array}{ll}
q_2,&{\rm if}\;x\in \beta_i(Z_i),\\
q_1,&{\rm if}\;x\in \beta_i(T_i).
\end{array}
\right.
\end{equation}
Since the automaton $A^{-1}$ is reversible, we obtain by (\ref{e1})--(\ref{e3}):
\begin{equation}\label{e4}
\alpha_i(Z_i)=\beta_i(Z_i),\;\;\;\alpha_i(T_i)=\beta_i(T_i).
\end{equation}

Let $A^{(i)}=(X^{(i)}, Q, \varphi^{(i)}, \psi^{(i)})$ ($i\geq 0$) be the $i$-shift of  the automaton $A$  and let us denote:
$$
a_i:=A^{(i-1)}_{q_1},\;\;\;b_i:=A^{(i-1)}_{q_2},\;\;\;i\geq 1.
$$
For each $i\geq 1$, by the definition of the mappings $A^{(i-1)}_{q_1}$ and $A^{(i-1)}_{q_2}$ and by the reversibility of the automaton $A^{(i-1)}$, we obtain for all $x\in X_{i}$ and  $w\in (X^{(i)})^*$ the following recursions:
\begin{eqnarray}
a_i(xw)&=&\left\{
\begin{array}{ll}
\alpha_i(x)a_{i+1}(w),&{\rm if}\;x\in Z_i,\\
\alpha_i(x)b_{i+1}(w),&{\rm if}\;x\in T_i,
\end{array}
\right.\label{e5}\\
b_i(xw)&=&\left\{
\begin{array}{ll}
\beta_i(x)b_{i+1}(w),&{\rm if}\;x\in Z_i,\\
\beta_i(x)a_{i+1}(w),&{\rm if}\;x\in T_i.
\end{array}\label{e6}
\right.
\end{eqnarray}
By the formulae (\ref{e1})--(\ref{e3}), we can write:
\begin{eqnarray}
a_i^{-1}(xw)&=&\left\{
\begin{array}{ll}
\alpha_i^{-1}(x)a_{i+1}^{-1}(w),&{\rm if}\;x\in \alpha_i(Z_i),\\
\alpha_i^{-1}(x)b_{i+1}^{-1}(w),&{\rm if}\;x\in \alpha_i(T_i),
\end{array}
\right.\label{e7}\\
b_i^{-1}(xw)&=&\left\{
\begin{array}{ll}
\beta_i^{-1}(x)b_{i+1}^{-1}(w),&{\rm if}\;x\in \beta_i(Z_i),\\
\beta_i^{-1}(x)a_{i+1}^{-1}(w),&{\rm if}\;x\in \beta_i(T_i).
\end{array}\label{e8}
\right.
\end{eqnarray}
Let us denote $c_i:=a_i^{-1}b_i$ ($i\geq 1$). Then by (\ref{e6})--(\ref{e7}), we obtain for all $x\in X_i$ and  $w\in (X^{(i)})^*$:
\begin{equation}\label{e9}
c_i(xw)=a_i^{-1}b_i(xw)=\left\{
\begin{array}{ll}
a_i^{-1}(\beta_i(x)b_{i+1}(w)),&{\rm if}\;x\in Z_i,\\
a_i^{-1}(\beta_i(x)a_{i+1}(w)),&{\rm if}\;x\in T_i.
\end{array}
\right.
\end{equation}
By (\ref{e4}), the condition $x\in Z_i$ can be equivalently written as $\beta_i(x)\in\alpha_i(Z_i)$. Similarly, we have: $x\in T_i$ if and only if $\beta_i(x)\in \alpha_i(T_i)$. Hence, we obtain by (\ref{e7}) and (\ref{e9}):
$$
c_i(xw)=\left\{
\begin{array}{ll}
\alpha_i^{-1}\beta_i(x)a_{i+1}^{-1}b_{i+1}(w),&{\rm if}\;x\in Z_i,\\
\alpha_i^{-1}\beta_i(x)b_{i+1}^{-1}a_{i+1}(w),&{\rm if}\;x\in T_i.
\end{array}
\right.
$$
If we now denote $\gamma_i:=\alpha_i^{-1}\beta_i\in Sym(X_i)$, then we can write:
$$
c_i(xw)=\left\{\begin{array}{ll}
\gamma_i(x)c_{i+1}(w),&{\rm if}\;x\in Z_i,\\
\gamma_i(x)c_{i+1}^{-1}(w),&{\rm if}\;x\in T_i.
\end{array}
\right.
$$
By (\ref{e4}), we see that $\gamma_i(Z_i)=Z_i$ and $\gamma_i(T_i)=T_i$. Consequently, for any integer $N$ we can write:
$$
c_i^{N}(xw)=\left\{\begin{array}{ll}
\gamma_i^{N}(x)c_{i+1}^N(w),&{\rm if}\;x\in Z_i,\\
\gamma_i^{N}(x)c_{i+1}^{-N}(w),&{\rm if}\;x\in T_i.
\end{array}
\right.
$$
In particular,  for any word $w=x_1x_2\ldots x_t\in X^*$ and any integer $N$, we obtain:
\begin{equation}\label{s1s}
c_1^N(w)=\gamma_1^{N_1}(x_1)\gamma_2^{N_2}(x_2)\ldots\gamma_t^{N_t}(x_t),
\end{equation}
where $N_i\in \{-N, N\}$ for every $1\leq i\leq t$.  Assume now that the changing alphabet $X$ is bounded. Then there is a positive integer $N$ such that $\gamma_i^N=\gamma_i^{-N}=id_{X_i}$ for every $i\geq 1$ (suffice it to take $N=r!$, where $r:=\sup_{i\geq 1}|X_i|$). Hence the element $c_1=a_1^{-1}b_1=A_{q_1}^{-1}A_{q_2}$ is of finite order in the group $G(A)=\langle a_1, b_1\rangle$. Moreover, if the element $c_1$ is trivial, then taking $N=1$ in (\ref{s1s}), we obtain: $N_i\in\{-1,1\}$ for $1\leq i\leq t$, and hence $\gamma_i=id_{X_i}$ for every $i\geq 1$, which implies $\alpha_i=\beta_i$ for $i\geq 1$. Consequently, by (\ref{e5})--(\ref{e6}), we have in this case: $a_i=b_i$ for all $i\geq 1$. In particular, for every word $w=x_1x_2\ldots x_t\in X^*$, we obtain $a_1(w)=\alpha_1(x_1)\alpha_2(x_2)\ldots\alpha_t(x_t)$.
Since $X$ is bounded, the last equality implies  that the generator $a_1$ is of finite order in this case. Summarizing, if $X$ is bounded, then  $G(A)$ can not be a non-trivial torsion-free group, and hence $G(A)$ can not be isomorphic to $\mathcal{F}_2$. Conversely, if $X$ is unbounded, then the existence of a $2$-state bi-reversible automaton  over $X$ which  generates $\mathcal{F}_2$ directly follows from the existence of the corresponding diagonal construction -- see Theorem~\ref{t0} (see also Examples~\ref{ex1}--\ref{ex2} in Section~\ref{sec2}). \qed
\end{proof}

Let $X=(X_i)_{i\geq 1}$ be an arbitrary unbounded changing alphabet. By using the construction of the automaton $A$  from Example~\ref{ex2} (see Section~\ref{sec2}), we may provide an explicit  construction of a 2-state bi-reversible automaton over $X$ which generates $\mathcal{F}_2$. To this end, we choose  a strictly increasing sequence $\xi:=(\xi_i)_{i\geq 1}$ of natural numbers such that the changing alphabet  $Y=(Y_i)_{i\geq 1}$ with  $Y_i:=X_{\xi_i}$  is unbounded and the sequence $(|Y_{i}|)_{i\geq 1}$ is non-decreasing and satisfies:  $|Y_i|\geq 2$ for every $i\geq 1$. Let  $A=(Y, \{q_1, q_2\}, \varphi, \psi)$ be the automaton from Example~\ref{ex2} and let $B=(X, \{q_1, q_2\}, \varphi', \psi')$ be an automaton which arises from $A$ as follows (below we denote ${\rm Im}(\xi):=\{\xi_i\colon i\geq 1\}$ and if $i\in {\rm Im}(\xi)$, then by $i'$ we denote a unique  $j\geq 1$ such  that $\xi_j=i$):
$$
\varphi'_i(q,x):=\left\{
\begin{array}{ll}
\varphi_{i'}(q,x),&{\rm if}\;i\in {\rm Im}(\xi),\\
q, &{\rm if}\;i\notin {\rm Im}(\xi),
\end{array}
\right.\psi'_i(q,x):=\left\{
\begin{array}{ll}
\psi_{i'}(q,x),&{\rm if}\;i\in {\rm Im}(\xi),\\
x, &{\rm if}\;i\notin {\rm Im}(\xi)
\end{array}
\right.
$$
for all $q\in \{q_1, q_2\}$, $i\geq 1$ and $x\in X_i$. Since the automaton $A$ is bi-reversible, the automaton $B$ is also bi-reversible.

\begin{theorem}
$G(B)\simeq \mathcal{F}_2$.
\end{theorem}
\begin{proof}
Let us choose an arbitrary word $w\in X^*$. Then we can uniquely choose the number $l\geq 0$ and the letters $y_i\in Y_i$ ($1\leq i\leq l$) such that $w=v_1y_1\ldots v_ly_lv_{l+1}$, where the sequences $v_i$ ($1\leq i\leq l$) satisfy: $|v_i|=\xi_i-\xi_{i-1}-1$ (we assume $\xi_0=0$). By  the definition of the automaton $B$ and the definition of the inverse automaton, we see that the actions of $B$ (resp. of $B^{-1}$) on the sequences $v_i$ ($1\leq i\leq l+1$) are trivial and $B$ (resp. $B^{-1}$) does not change its state when reading each of these sequences, whereas the actions of $B$ (resp. of $B^{-1}$) on the letters $y_i$ ($1\leq i\leq l$) coincide with the corresponding actions of the automaton $A$ (resp. of the automaton $A^{-1}$). Thus, if we denote $\widetilde{w}:=y_1\ldots y_l$ and if for some $q\in Q$ and $\eta\in\{-1,1\}$ we have $A_q^\eta(\widetilde{w})=y_1'\ldots y'_l$, where  $y'_i\in Y_i$ ($1\leq i\leq l$), then $B_q^\eta(w)=v_1y_1'\ldots v_ly_l'v_{l+1}$. Consequently, for any sequences $(q_1, q_2, \ldots, q_m)\in Q^m$ and $(\eta_1, \eta_2, \ldots, \eta_m)\in\{-1,1\}^m$ ($m\geq 1$)  and for any letters $y'_i\in Y_i$ ($1\leq i\leq l$) we obtain the following implication:
$$
A_{q_1}^{\eta_1}A_{q_2}^{\eta_2}\ldots A_{q_m}^{\eta_m}(\widetilde{w})=y_1'y_2'\ldots y'_l\Rightarrow B_{q_1}^{\eta_1}B_{q_2}^{\eta_2}\ldots B_{q_m}^{\eta_m}(w)=v_1y_1'\ldots v_ly_l'v_{l+1}.
$$
Thus,  since the mapping $X^*\ni w\mapsto \widetilde{w}\in Y^*$ is surjective, we obtain:
$$
A_{q_1}^{\eta_1}A_{q_2}^{\eta_2}\ldots A_{q_m}^{\eta_m}=id_{Y^*}\Leftrightarrow B_{q_1}^{\eta_1}B_{q_2}^{\eta_2}\ldots B_{q_m}^{\eta_m}=id_{X^*}
$$
Hence, $B_q\mapsto A_q$ ($q\in Q$) induces the isomorphisms $G(B)\simeq G(A)\simeq \mathcal{F}_2$.\qed
\end{proof}

\section{The classification  in the case of the sequence of binary alphabets}

In this section, we completely characterize the class $BIR_{2,2}$ of all groups generated by a 2-state bi-reversible automaton over the sequence of binary alphabets. We show that $BIR_{2,2}$ consists exactly of five  groups,  and  only three of them can be generated by a 2-state bi-reversible Mealy automaton over the binary alphabet.

\begin{theorem}
The class $BIR_{2,2}$ of all groups generated by a 2-state bi-reversible automaton over the sequence of binary alphabets consists of five groups: the trivial group $\{id\}$, the cyclic group of order two $\mathbb{Z}_2$, the Klein group $\mathbb{Z}_2\times\mathbb{Z}_2$,  the cyclic group of order four $\mathbb{Z}_4$ and the direct product $\mathbb{Z}_2\times\mathbb{Z}_4$. Only the first three groups belong to the subclass $BIRM_{2,2}\subseteq BIR_{2,2}$ of  groups  generated by a 2-state bi-reversible Mealy automaton over the binary alphabet.
\end{theorem}
\begin{proof}
The second part directly follows from the well-known classification of the groups generated by a 2-state Mealy automaton over the binary alphabet (the corresponding constructions are described in Section 4.1 in~\cite{9} by the popular language of wreath recursions). For the first part, let us consider an arbitrary 2-state bi-reversible automaton $A=(X, \{q_1, q_2\}, \varphi, \psi)$ over the changing alphabet $X=(X_i)_{i\geq 1}$ in which $X_i:=\{0,1\}$ for every $i\geq 1$. For each $i\geq 0$ let $A^{(i)}=(X, Q, \varphi^{(i)}, \psi^{(i)})$ be the $i$-shift of $A$. Let us denote
$$
a_i:=A^{(i-1)}_{q_1},\;\;\;b_i:=A^{(i-1)}_{q_2},\;\;\;i\geq 1.
$$

\begin{lemma}\label{lem1}
For every $i\geq 1$ and every word $w\in X^*$ the following equalities hold:
$$
a_ib_i(w)=b_ia_i(w),\;\;\; a_i^2(w)=b_i^2(w),\;\;\;a_i^4(w)=b_i^4(w)=w.
$$
\end{lemma}
\begin{proof}[of Lemma~\ref{lem1}]
We use induction on the length of $w$. Let us assume that the claim holds for all $i\geq 1$ and for all words $w\in X^*$ with the length $|w|=t$ for some fixed $t\geq 0$. Let us fix $i\geq 1$. By the bi-reversibility of $A$, we see that  there are $\pi_1, \pi_2\in Sym(\{0,1\})$ and $c, d\in\{a_{i+1}, b_{i+1}\}$ such that $(c, d)\in\{(a_{i+1}, b_{i+1}), (b_{i+1},a_{i+1})\}$ and for every $x\in \{0,1\}$ and $w\in X^*$ one of the following three cases holds:
$$
a_i(xw)=\pi_1(x)c(w),\;\;\;b_i(xw)=\pi_2(x)d(w),
$$
or
$$
a_i(xw)=\left\{
\begin{array}{ll}
0c(w),&{\rm if}\;x=0,\\
1d(w),&{\rm if}\;x=1,\\
\end{array}
\right.\;\;\;b_i(xw)=\left\{
\begin{array}{ll}
0d(w),&{\rm if}\;x=0,\\
1c(w),&{\rm if}\;x=1,\\
\end{array}
\right.
$$
or
$$
a_i(xw)=\left\{
\begin{array}{ll}
1c(w),&{\rm if}\;x=0,\\
0d(w),&{\rm if}\;x=1,\\
\end{array}
\right.\;\;\;b_i(xw)=\left\{
\begin{array}{ll}
1d(w),&{\rm if}\;x=0,\\
0c(w),&{\rm if}\;x=1.\\
\end{array}
\right.
$$
If we denote $C:= cd$ and $D:=dc$, then we obtain in the first case:
\begin{eqnarray*}
a_ib_i(xw)=\pi_1\pi_2(x)C(w),\;\;\;b_ia_i(xw)=\pi_1\pi_2(x)D(w),\\
a_i^2(xw)=xc^2(w),\;\;\;b_i^2(xw)=xd^2(w),\\
a_i^4(xw)=xc^4(w),\;\;\;b_i^4(xw)=xd^4(w).
\end{eqnarray*}
In the second case, we obtain
\begin{eqnarray*}
a_ib_i(xw)=\left\{
\begin{array}{ll}
0C(w),&{\rm if}\;x=0,\\
1D(w),&{\rm if}\;x=1,\\
\end{array}
\right.\;
b_ia_i(xw)=\left\{
\begin{array}{ll}
0D(w),&{\rm if}\;x=0,\\
1C(w),&{\rm if}\;x=1,\\
\end{array}
\right.\\
a_i^2(xw)=\left\{
\begin{array}{ll}
0c^2(w),&{\rm if}\;x=0,\\
1d^2(w),&{\rm if}\;x=1,\\
\end{array}
\right.\;\;\;b_i^2(xw)=\left\{
\begin{array}{ll}
0d^2(w),&{\rm if}\;x=0,\\
1c^2(w),&{\rm if}\;x=1,\\
\end{array}
\right.\\
a_i^4(xw)=\left\{
\begin{array}{ll}
0c^4(w),&{\rm if}\;x=0,\\
1d^4(w),&{\rm if}\;x=1,\\
\end{array}
\right.\;\;\;b_i^4(xw)=\left\{
\begin{array}{ll}
0d^4(w),&{\rm if}\;x=0,\\
1c^4(w),&{\rm if}\;x=1.\\
\end{array}
\right.\\
\end{eqnarray*}
In the third case, we obtain:
\begin{eqnarray*}
a_ib_i(xw)=\left\{
\begin{array}{ll}
0d^2(w),&{\rm if}\;x=0,\\
1c^2(w),&{\rm if}\;x=1,\\
\end{array}
\right.\;
b_ia_i(xw)=\left\{
\begin{array}{ll}
0c^2(w),&{\rm if}\;x=0,\\
1d^2(w),&{\rm if}\;x=1,\\
\end{array}
\right.\\
a_i^2(xw)=\left\{
\begin{array}{ll}
0D(w),&{\rm if}\;x=0,\\
1C(w),&{\rm if}\;x=1,\\
\end{array}
\right.\;\;\;b_i^2(xw)=\left\{
\begin{array}{ll}
0C(w),&{\rm if}\;x=0,\\
1D(w),&{\rm if}\;x=1,\\
\end{array}
\right.\\
a_i^4(xw)=\left\{
\begin{array}{ll}
0D^2(w),&{\rm if}\;x=0,\\
1C^2(w),&{\rm if}\;x=1,\\
\end{array}
\right.\;\;\;b_i^4(xw)=\left\{
\begin{array}{ll}
0C^2(w),&{\rm if}\;x=0,\\
1D^2(w),&{\rm if}\;x=1.\\
\end{array}
\right.
\end{eqnarray*}
By the inductive assumption, since $(c, d)\in\{(a_{i+1}, b_{i+1}), (b_{i+1},a_{i+1})\}$, we obtain for every word $w\in X^*$ of length $t$:
$C(w)=D(w)$, $c^2(w)=d^2(w)$, $c^4(w)=d^4(w)=w$. Consequently, since $|c(w)|=|w|=t$, we obtain $C^2(w)=CD(w)=cddc(w)=cd^2(c(w))=cc^2(c(w))=c^4(w)=w$. Similarly $D^2(w)=w$. Thus in each case, we have
$$
a_ib_i(xw)=b_ia_i(xw),\;\;\;a_i^2(xw)=b_i^2(xw),\;\;\;a_i^4(xw)=b_i^4(xw)=xw.
$$
An inductive argument finishes the proof of Lemma~\ref{lem1}.\qed
\end{proof}

By Lemma~\ref{lem1}, we obtain that the group $G(A)=\langle a_1, b_1\rangle$ is abelian and the generators $a_1$, $b_1$ satisfy:  $a_1^2=b_1^2$ and $a_1^4=b_1^4=id$. Since the direct product $\mathbb{Z}_2\times \mathbb{Z}_4$ has the presentation $\langle a, b\colon a^4=b^4=1, a^2=b^2, ab=ba\rangle$, the group $G(A)$ is a quotient of this product. There are only five  quotients of $\mathbb{Z}_2\times \mathbb{Z}_4$: the trivial group, $\mathbb{Z}_2$, $\mathbb{Z}_2\times\mathbb{Z}_2$, $\mathbb{Z}_4$ and $\mathbb{Z}_2\times \mathbb{Z}_4$. As we mentioned above,  the first three groups were proved to be  generated by a suitable Mealy automaton. We now construct the corresponding automata for the groups $\mathbb{Z}_2\times \mathbb{Z}_4$ and $\mathbb{Z}_4$.

Let $A=(X, \{q_1, q_2\}, \varphi, \psi)$ be an automaton in which  the transition and the  output functions are defined for all $i\geq 1$, $q\in Q$, $x\in \{0,1\}$ as follows:
$$
\varphi_i(q,x):=\left\{
\begin{array}{ll}
\overline{q},&{\rm if}\;x=1\;{\rm and}\;2\nmid i,\\
q,&{\rm otherwise},
\end{array}
\right.\;\psi_i(q,x):=\left\{
\begin{array}{ll}
x,&{\rm if}\;q=q_2\;{\rm and}\;2\mid i,\\
\tau(x),&{\rm otherwise},
\end{array}
\right.
$$
where  $\tau\in Sym(\{0,1\})$ is a transposition, $\overline{q}_1:=q_2$, $\overline{q}_2:=q_1$. It can be easily verified that the automaton $A$ is bi-reversible and the mappings $a_i:=A^{(i-1)}_{q_1}$, $b_i:=A^{(i-1)}_{q_2}$ ($i\geq 1$) satisfy for all $x\in\{0,1\}$ and $w\in \{0,1\}^*$ the following recursions: if $i$ is odd, then
\begin{equation}\label{ee1}
a_{i}(xw)=\left\{
\begin{array}{l}
1a_{i+1}(w),\;{\rm if}\;x=0,\\
0b_{i+1}(w),\;{\rm if}\;x=1,\\
\end{array}
\right.\;
b_{i}(xw)=\left\{
\begin{array}{l}
1b_{i+1}(w),\;{\rm if}\;x=0,\\
0a_{i+1}(w),\;{\rm if}\;x=1,
\end{array}
\right.
\end{equation}
and if $i$ is even, then
\begin{equation}\label{ee2}
a_{i}(xw)=\tau(x)a_{i+1}(w),\;\;\;b_{i}(xw)=xb_{i+1}(w).
\end{equation}
By (\ref{ee2}), we have $a_{2i}\neq b_{2i}$ ($i\geq 1$), and hence, by (\ref{ee1}), we obtain:  $a_i\neq b_i$ for every $i\geq 1$. By (\ref{ee1})--(\ref{ee2}) and by  the relations $a_ib_i=b_ia_i$, $a_i^2=b_i^2$ ($i\geq 1$), we  obtain:
\begin{eqnarray*}
a_{2i-1}b_{2i-1}(xw)=xa_{2i}^2(w),\;\;\;a_{2i}b_{2i}(xw)=\tau(x)a_{2i+1}b_{2i+1}(w),\\
a_{2i}^2(xw)=xa_{2i+1}^2(w),\;\;\;\;\;a_{2i-1}^2(xw)=xa_{2i}b_{2i}(w)
\end{eqnarray*}
for all $i\geq 1$, $x\in \{0,1\}$ and $w\in \{0,1\}^*$. In particular $a_4b_4\neq id$ and $a_2b_2\neq id$. Hence $a_1^2\neq id$ and $a_3^2\neq id$, which implies $a_2^2\neq id$. Consequently $a_1b_1\neq id$. Next, since $a_1\neq b_1$ and $a_1^4=id$, we  obtain:
$$
a_1^2\neq a_1b_1,\;a_1^2\neq a_1^3b_1,\;a_1b_1\neq a_1^3b_1,\;a_1^3b_1\neq id.
$$
Thus $S:=\{id, a_1^2, a_1b_1, a_1^3b_1\}$ is a four-element subset of the group $G(A)$. Since each element in $S$  fixes the first letter of any word and $a_1(0)=1$,  we have $a_1\notin S$, and hence the group $G(A)=\langle a_1, \ b_1\rangle$ contains  at least five elements. Thus it must be $G(A)\simeq \mathbb{Z}_2\times \mathbb{Z}_4$.

Finally, let $A'=(X, \{q_1, q_2\}, \varphi, \psi')$ be an automaton which arises from the  automaton $A$ by replacing the labelings $\sigma_{q,i}\colon x\mapsto \psi_i(q,x)$ with $q\in \{q_1, q_2\}$ and $i\geq 3$ by the trivial permutations. Obviously, the automaton $A'$ is bi-reversible. If we now denote $a_i:=A'^{(i-1)}_{q_1}$ and $b_i:=A'^{(i-1)}_{q_2}$ ($i\geq 1$), then we have: $a_i=b_i=id$ for $i\geq 3$. For all $x\in\{0,1\}$ and $w\in\{0,1\}^*$ we also have:
\begin{equation}\label{ee3}
a_2(xw)=\tau(x)a_3(w)=\tau(x)w,\;\;\;b_2(xw)=xb_3(w)=xw
\end{equation}
and
\begin{equation}\label{ee4}
a_{1}(xw)=\left\{
\begin{array}{ll}
1a_{2}(w),&{\rm if}\;x=0,\\
0b_{2}(w),&{\rm if}\;x=1,\\
\end{array}
\right.\;
b_{1}(xw)=\left\{
\begin{array}{ll}
1b_{2}(w),&{\rm if}\;x=0,\\
0a_{2}(w),&{\rm if}\;x=1.
\end{array}
\right.
\end{equation}
By (\ref{ee3}), we have:  $a_2\neq id$, $b_2=id$, and hence, by (\ref{ee4}), we obtain: $a_1^2(xw)=xa_2(w)$ for $x\in\{0,1\}$, $w\in\{0,1\}^*$. Thus $a_1^2\neq id$, which means that  $a_1$ is of order four in the group $G(A')=\langle a_1, b_1\rangle$. Since $a_2^2=b_2^2=id$, we obtain by (\ref{ee4}): $a_1b_1(xw)=xw$ for all $x\in\{0,1\}$ and $w\in\{0,1\}^2$. Thus $b_1=a_1^{-1}$ and $G(A')=\{id, a_1, a_1^2, a_1^3\}$. Consequently $G(A')\simeq  \mathbb{Z}_4$.
\qed
\end{proof}

The situation  changes if we consider the wider class $IR_{2,2}$ consisting  of  groups generated by a 2-state invertible-reversible automaton over the sequence of binary alphabets, as well as the class $BIR_{2,3}$ of  groups  generated by a 2-state bi-reversible automaton over the sequence of ternary alphabets $X_i=\{0,1,2\}$ ($i\geq 1$). For the next result, we use  some well-known  constructions of Mealy automata  to show that both these classes  contain infinitely many pairwise non-isomorphic finite groups.

\begin{theorem}
The class $IR_{2,2}$ of  groups generated by a 2-state invertible-reversible automaton over the sequence of binary alphabets, as well as the class $BIR_{2,3}$ of  groups  generated by a 2-state bi-reversible automaton over the sequence of ternary alphabets contains infinitely many pairwise non-isomorphic finite groups.
\end{theorem}
\begin{proof}
For every automaton $A=(X, Q, \varphi, \psi)$ and for  every $k\geq 0$ let $A|^k:=(X, Q, \varphi',\psi')$ be the $k$-th restriction   of $A$, i.e. the automaton in which the transition and the output functions are defined as follows: $\varphi'_i=\varphi_i$, $\psi'_i=\psi_i$ for $1\leq i\leq k$ and $\varphi'_i(q,x)=q$, $\psi'_i(q,x)=x$ for $i>k$, $q\in Q$ and $x\in X_i$. Obviously, if $A$ is invertible (resp. reversible, bi-reversible), then so is  $A|^k$. For all $q\in Q$ and $k\geq 0$ we also have: if $w\in X^*$ is any word of length $|w|\geq k$  and $w=vu$ for some $v\in X^k$ and $u\in (X^{(k)})^*$, then $A|^k_{q}(w)=A|^k_q(vu)=A_q(v)u$. Thus the mapping $A|^k_q\colon X^*\to X^*$ can be identified with the restriction of $A_q$ to the set $X^k$. In particular,  if $A$ is invertible, then all the groups $G(A|^k)$ ($k\geq 0$) are finite and the mapping
$$
A_q\mapsto (A|^k_q)_{k\geq 0},\;\;\;q\in Q,
$$
induces an embedding of the group $G(A)$ into the Cartesian product $\prod_{k\geq 0}G(A|^k)$. If the group $G(A)$ is infinite, then there exist infinitely many pairwise non-isomorphic finite groups among the groups $G(A|^k)$ for $k\geq 0$. Otherwise, the product  $\prod_{k\geq 0}G(A|^k)$ would be  a locally finite group; and consequently the group $G(A)$ would be finite.

Let now $A_1$ be the 2-state invertible-reversible Mealy automaton over the binary  alphabet  which generates the lamplighter group $\mathbb{Z}_2\wr\mathbb{Z}$ (for the exact definition of $A_1$ see Section 4.1 in~\cite{9}). Also,  let $A_2$ be the 2-state bi-reversible Mealy automaton  over the ternary alphabet which is the dual to the 3-state Bellaterra automaton (for the construction of $A_2$ and some properties of the group $G(A_2)$ see Section~1.10.3 in~\cite{10}).  Our claim now follows  from the well-known fact that both the groups $G(A_1)$ and $G(A_2)$ are infinite.\qed
\end{proof}

\begin{remark}
Note that for all $n,d\geq 1$  the number of  $n$-state invertible Mealy automata over a $d$-letter alphabet is finite. In particular,
both the classes $IRM_{2,2}\subseteq IR_{2,2}$ and $BIRM_{2,3}\subseteq BIR_{2,3}$ of  groups generated by the corresponding Mealy automata are finite. On the other hand, if $(n,d)\neq (1,1)$, then there are uncountably many  $n$-state invertible (resp. bi-reversible) automata over a $d$-letter alphabet. In \cite{20}, we completely characterize for every $n\geq 1$ all abelian groups from  the class $I_{n,2}$ of groups generated by an $n$-state invertible automaton over the sequence of binary alphabets. In particular, we obtained that the free abelian groups $\mathbb{Z}$, $\mathbb{Z}\times \mathbb{Z}$ and the direct products $\mathbb{Z}\times \mathbb{Z}_{2^m}$ with $m\geq 1$ belong to $I_{2,2}$. Hence $I_{2,2}$ contains infinitely many pairwise non-isomorphic infinite groups. However, the constructed automata are not reversible. We are interested in the following questions: is the class $I_{2,2}$ (resp. the subclass $IR_{2,2}\subseteq I_{2,2}$) countable or uncountable? Is there only finitely many infinite groups in $IR_{2,2}$? Which non-abelian groups belong to $IR_{2,2}$ (or to $I_{2,2}$)? In particular, is it true that  $\mathcal{F}_2\in I_{2,2}$?
\end{remark}

\end{document}